\documentclass[reqno,12pt]{amsart}
\usepackage{bbm}
\usepackage[all,pdf]{xy}
\usepackage{epsfig}
\usepackage{amsmath}
\usepackage{amssymb}
\usepackage{amscd}
\usepackage{graphicx}
\newtheorem{thm}{Theorem}[section]
\newtheorem{lemma}[thm]{Lemma}

\newtheorem{ques}[thm]{Question}
\newtheorem{cor}[thm]{Corollary}
\newtheorem{defn}[thm]{Definition}
\newtheorem{rem}[thm]{Remark}

\def \la {\longrightarrow}

\def \ol {\overline}
\def \N {\mathbb N}

\def \Z {\mathbb Z}

\numberwithin{equation}{section}

\begin{document}

	\baselineskip 14pt
	
	\title[Weak mean equicontinuity]{Weak mean equicontinuity for  a countable discrete amenable group action}
	
	\date{Dec. 24, 2020}
	
	\author{Leiye Xu and Liqi Zheng}
	
	\address{Department of Mathematics, University of Science and Technology of
		China, Hefei, Anhui, 230026, P.R. China}
	
	\email{leoasa@mail.ustc.edu.cn,lqzheng@ustc.edu.cn}

\subjclass[2010]{Primary: 54H20; Secondary: 37A20, 37B05, 37B45.}
% Please provide minimum  5 keywords.
\keywords{Wasserstein distance, Unique ergodicity, Amenable group, Weak mean equicontinuity.}

 \begin{abstract}
		%Let $(X,G)$ be a topological dynamical system and $\mathcal{F} = \{F_n \}_{n=1}^{\infty}$ be a F{\o}lner sequence of countable discrete amenable group $G$. It is shown that $\lim_{ n \to +\infty} \inf_{ h \in Aut (F_n) }\frac 1{|F_n|} \sum_{g \in F_n}  d (gx , g^h y)$ exists when $x,y$ are generic points with respect to $\mathcal{F}$, where $Aut(F_n)$ is the set of all bijections from $F_n$ to itself. 
	    The weak mean equicontinuous properties for a countable discrete amenable group $G$ acting continuously on a compact metrizable space $X$ are studied. 
	    It is shown that the weak mean equicontinuity of $(X \times X,G)$ is equivalent to   the mean equicontinuity of $(X,G)$. Moreover, when $(X,G)$ has full  measure center or $G$ is abelian, it is shown that $(X,G)$ is weak mean equicontinuous if and only if all points in $X$ are uniquely  ergodic points and the map $x \to \mu_x^G$ is continuous, where $\mu_x^G$ is the unique ergodic measure on $\{\ol{Orb(x)}, G\}$.
	\end{abstract}
		\maketitle
		
	\section{Introduction}
	 Throughout this paper,  $G$ is a countable infinite discrete amenable group. By a \emph{$G$-system} we mean a pair $(X,G)$, where $X$ is a compact metric space with metric $d$ and $G$ acts as a group of continuous maps  from $X$ to itself.
	In particular, when $G=\mathbb{Z}$ the $\mathbb{Z}$-system $(X,\mathbb{Z})$ can be considered inducing by a homeomorphism  $T$ from $X$ to itself. In this case we say that $(X,T)$ is a \emph{topological dynamical system} (t.d.s. for short).

	In the theory of topological dynamical systems, several notions of continuity have been studied. People firstly focused on {\it equicontinuous systems}, because they have simple dynamical behaviors \cite{EJK,PJ}. In recent years, {\it mean equicontinuity }has received keen interest. We refer to \cite{FG,GM,WJLX,JL,JSX} for further study on mean equicontinuity and related subjects. But in the measure theoretic point of view, we are only interested in the cumulative effect of points in orbits, in which the order makes no sense. To ignore the order, Zheng and Zheng \cite{ZZ} introduced the notion {\it weak mean equicontinuity}. They show that a t.d.s. is weak mean equicontinuous if and only if the time average of any continuous function  converges to a continuous function.
	
	The aim of this paper is to study the weak mean equicontinuity for $G$-actions. In order to make clear statement of our results, we introduce the following notations. 	
	
	A countable infinite discrete group $G$ is \emph{amenable} \cite{EF} if there exists a sequence of finite nonempty subsets $\{F_n \}_{n=1}^{\infty} $ of $G$, which is called a \emph{left F{\o}lner sequence}, such that $\lim_{n \to +\infty} \frac {|gF_n \Delta F_n|}{|F_n|} = 0$ for every $g \in G$, where $| \cdot |$ denotes the cardinality of a set. 
	Every abelian group is amenable. For example, if $G = \mathbb{Z}^d$, then $\{F_n\}_{n=1}^{\infty}$ is a left F{\o}lner sequence of $G$, where $F_n = [-n, n]^d \cap \Z^d$ for all $n \in \N^+$.

	Denote by $M(X,G)$ the set of all $G$-invariant  Borel probability measures. It is well known that for a countable infinite amenable group $G$, $M(X,G)$ is not empty. In particular, we say $(X,G)$ is uniquely ergodic if $M(X,G)$ is singleton.
	
Given a $G$-system $(X,G)$, a finite nonempty subset $F$ of $G$ and $x \in X$, define 
\[
\mu_{x,F} = \frac 1{|F|} \sum_{g \in F} \delta_{gx}
\]
where $\delta_{z}$ is the Dirac measure which has full measure on point $z$. For a left F{\o}lner sequence $\mathcal{F} =\{F_n\}_{n=1}^{\infty}$ of $G$ and $x,y \in X$, define
\begin{equation*}
W_\mathcal{F}(x,y)= \limsup_{ n \to +\infty} W (\mu_{x,F_n}, \mu_{y,F_n}) ,
\end{equation*}
where $W(\cdot, \cdot)$ is  the{ \it Wasserstein distance} (see Section \ref{2} for definition) on Borel probability measure space. Since $W(\cdot, \cdot)$ satisfies triangle inequality, we have that $W_\mathcal{F}(x,y)$ also satisfies triangle inequality. Hence it is a pseudometric on $X$.

\begin{defn}
	Let $(X,G)$ be a $G$-system and $\mathcal{F}$ be a left F{\o}lner sequence of $G$.  We say $(X,G)$ is weak mean equicontinuous with respect to $\mathcal{F}$ if for any $\varepsilon > 0$, there is $\delta >0$ such that  $W_\mathcal{F}(x,y) < \varepsilon$  whenever $x,y \in X$ with $d(x,y) < \delta$.  Specially, we call $(X,G)$ is weak mean equicontinuous if it is weak mean equicontinuous with respect to all  left F{\o}lner sequences of $G$.
\end{defn}
\begin{rem}
As in \cite{ZZ}, we can  define $W_{\mathcal{F}}(\cdot,\cdot)$ by ignoring the order. In Appendix, we will show that the two definitions are the same.
\end{rem}
  Let $\mathcal{F} = \{ F_n \}_{n=1}^{+\infty}$ be a left F{\o}lner sequence of $G$.  Then $(X,G)$ is weak mean equicontinuous with respect to $\mathcal{F}$ if and only if $W_\mathcal{F}(x,y)$ is continuous on $X \times X$. Recall that $(X,G)$ is mean equicontinuous with respect to $\mathcal{F}$  if for any $\varepsilon > 0$, there is $\delta >0$ such that 
	 $$\limsup_{n \to +\infty} \frac 1{|F_n|} \sum_{g\in F_n} d(gx, gy) < \varepsilon$$
	 whenever $x,y \in X$ and $d(x,y) < \delta$ \cite{GM}. Hence, the relation of mean equicontinuity and weak mean equicontinuity can be stated as follows.
	 		
    \begin{thm}\label{thm1.2}
    	Let $(X,G)$ be a $G$-system and $\mathcal{F}$ be a left F{\o}lner sequence of $G$. Then the following two statements are equivalent:
    	\begin{itemize}
    		\item [(\emph{1})]
    		$(X,G)$ is mean equicontinuous  with respect to $\mathcal{F}$;
    		
    		\item [(\emph{2})]
    		$(X\times X, G)$ is weak mean equicontinuous  with respect to $\mathcal{F}$.
    	   	\end{itemize}
       
    \end{thm}
     	Here  $(X \times X, G)$ is the product system of $(X,G)$ such that 
     $g(x,y) = (gx, gy)$
     holds for any $g \in G$ and $(x,y) \in X \times X$. The metric $\tilde d$  on $X\times X$ is defined by
     $\tilde d((x,y),(x',y'))=d(x,x')+d(y,y')$ for $x,x',y,y'\in X$.

     For a $G$-system $(X,G)$, the {\it measure center } is $\ol{\cup_{\mu \in M(X,G ) } \text{supp} \mu} $, which is the minimal compact subset of $X$ with full measure for any invariant measure $\mu \in M(X,G)$. We say $x \in X$ is  a { \it uniquely ergodic point } if the $G$-system $( \overline{ Orb(x) }, G) $ is uniquely ergodic, where $Orb(x)=\{gx:g\in G\}$ is the orbit of $x$. If $x \in X$ is uniquely ergodic point, we note the unique ergodic measure of $( \overline{  Orb(x)}, G) $ by $\mu_x^G$. Recently, Frantzikinakis and Host \cite{FB} prove that in a zero entropy system, all uniquely ergodic points satisfy the logarithmic Sarnak conjecture.
     Thus it is interesting to consider the
     case when all the points in a $G$-system are uniquely ergodic points. In this paper, we prove that in a weak mean equicontinuous system, all the points in the measure center are uniquely ergodic points.
     
    We also focus on uniform convergence, a classic topic of dynamical systems for it  meets with more favor  in mathematical physics. In 1952, Oxtoby \cite{JCO} shows that the Birkhoff average of any continuous function converges uniformly in uniquely ergodic systems. In 1982, Johnson and Moser \cite{JM} prove that for a t.d.s. $(X,T)$, if a continuous  function $f \in C(X)$ is orthogonal to all invariant measures (i.e. $\int f \mathrm{d} \mu = 0$ holds for any $\mu \in M(X,T)$), then the  Birkhoff average converges to $0$ uniformly. Recently, Zhang and Zhou \cite{ZZH} also study related topics. In this paper, we show that for weak mean equicontinuous systems,  Birkhoff averages of  continuous functions are uniformly convergent on measure center.

   \begin{thm}\label{thm1.4}
   	Let $(X,G)$ be a $G$-system.  If the measure center of $(X,G)$ is $X$, then the following statements are equivalent.
   	\begin{itemize}
   		
   		\item [(\emph{1})]
   		$(X,G)$ is weak mean equicontinuous;
   		\item [(\emph{2})]
   		$(X,G)$ is weak mean equicontinuous with respect to a F{\o}lner sequence of $G$;		
   		\item [(\emph{3})]
   		All points in $X$ are uniquely  ergodic points and the map $x\to \mu_{x}^G$ is continuous;
   		\item [(\emph{4})]For any $f\in C(X)$, there exists $f^*\in C(X)$ such that  for any left  F{\o}lner sequence $\mathcal{F}=\{F_n\}$ of $G$,
   		$$\frac{1}{|F_n|}\sum_{g\in F_n}f(gx) \text{ converges to } f^*(x)\text{ uniformly}.$$
   	\end{itemize}
   \end{thm}

   By Theorem \ref{thm1.4}, we have the following.
   \begin{cor}
   	Let $(X,G)$ be a minimal $G$-system.  Then the following are equivalent:
   	\begin{itemize}
   		
   		\item [(\emph{1})]
   		$(X,G)$ is weak mean equicontinuous.		
   		\item [(\emph{2})]
   		All points in $X$ are uniquely  ergodic points and the map $x\to \mu_{x}^G$ is continuous;
   		\item[(\emph{3})] 
   		$(X,G)$ is uniquely ergodic.
   	\end{itemize}
   \end{cor}

     \begin{ques}Is the condition  “If the measure center of $(X,G)$ is $X$” necessary in Theorem \ref{thm1.4}?
     \end{ques}
  We call  a sequence of finite nonempty subsets $\{L_n \}_{n=1}^{\infty} $ of $G$ is a \emph{right F{\o}lner sequence}  if $\lim_{n \to +\infty} \frac {|L_n \Delta L_n g|}{|L_n|} = 0$ for every $g \in G$. It is easy to check that $\{L_n \}_{n=1}^{\infty} $ is a right F{\o}lner sequence if and only if 
     $\{L_n^{-1} \}_{n=1}^{\infty} $ is a left F{\o}lner sequence, where $L_n^{-1} = \{g^{-1}: g \in L_n\}$ for all $n \in \N^+$.
    \begin{thm}\label{thm1.5}
    	Let $(X,G)$ be a $G$-system and $\mathcal{L}$ be a left and  right F{\o}lner sequence of $G$.  Then the following statements are equivalent.
    	\begin{itemize}
    		
    		\item [(\emph{1})]
    		$(X,G)$ is weak mean equicontinuous;
    		    		\item [(\emph{2})]
    		$(X,G)$ is weak mean equicontinuous with respect to $\mathcal{L}$;		
    		\item [(\emph{3})]
    	All points in $X$ are uniquely ergodic points and the map $x\to \mu_{x}^G$ is continuous; 
    	\item [(\emph{4})]For any $f\in C(X)$, there exists $f^*\in C(X)$ such that  for any left F{\o}lner sequence $\mathcal{F}=\{F_n\}$ of $G$,
    	$$\frac{1}{|F_n|}\sum_{g\in F_n}f(gx) \text{ converges to }  f^*(x)\text{ uniformly}.$$
    	\end{itemize}
    \end{thm}

    \begin{rem}
    	\begin{itemize}
    		\item [(1)]
    		 If $G$ is a finite extension of a countable discrete abelian group, then every left F{\o}lner sequence of $G$ is a right F{\o}lner sequence (see for example \cite{MA}, Section 1.1).Thus, the gruop which satisfies the condition of Theorem \ref{thm1.5} is not need to be abelian group.
    		
    		\item [(2)]
    		For some amenable group $G$, there exist left F{\o}lner sequences of $G$ which are not right F{\o}lner sequences (see for example \cite{MA}, Section 1.1). 
    	\end{itemize}
    \end{rem}

	 This paper is organized as follows. In Section \ref{2}, we introduce some basic notions and results needed in the paper. In Section \ref{4}, we prove Theorem \ref{thm1.2}. In Section \ref{5}, we prove Theorem \ref{thm1.4} and Theorem \ref{thm1.5}. In Appendix \ref{A}, we show that our definition is the same as in \cite{ZZ}.

	\section{Preliminaries}\label{2}
	
	In this section we recall some notions and results of $G$-systems which are needed in our paper. Note that $\N^+$ denotes the set of all positive integers in this paper.

	\subsection{Measures}
	Suppose that $(X,G)$ is a $G$-system. Denote by $\mathcal{B}(X)$ the Borel $\sigma$-algebra of $X$ and $M(X)$ the set of all Borel probability measures on $X$. For $\mu\in M(X)$, denote by $\text{supp} \mu$ the support of $\mu$, i.e. the smallest closed subset $W\subset X$ such that $\mu(W) = 1$. In the weak$^*$-topology, $M(X)$ is a nonempty compact convex space.
	
	\subsubsection{\bf Metrics on measure space}
	 Given $\mu, \nu \in M(X)$, a transport plan from $\mu$ to $\nu$ is a probability measure $\pi$ on the product space $X \times X$ such that 
	\begin{equation*}
		(p_1)_* \pi = \mu \ \ \text{and} \ \ ( p_2 )_* \pi = \nu,
	\end{equation*}
	where $p_1, p_2 : X \times X \to X $ are the canonical projections, and $(p_i)_* \pi = \pi \circ p_i^{-1}$ for $i=1,2$. Let $\Pi (\mu, \nu) $ denote the set of all transport plans from $\mu$ to $\nu$. Then the Wasserstein distance between $\mu$ and $\nu$ is defined as:
	\begin{equation*}
		W(\mu, \nu) = \inf_{\pi \in \Pi (\mu, \nu) } \int d(x,y) \mathrm{d} \pi (x,y).
	\end{equation*}
	The integral in the above formula is called the cost of the transport plan $\pi$, and the infimum is always attained.  The Wasserstein distance is a metric on $M(X)$ and the induced topology is just the weak*-topology on $M(X)$ (see for example \cite{CV},Theorem 7.3 and Theorem 7.12).

	Let  $\{f_n\}_{n=1}^{\infty}$ be a countable dense subset of continuous functions space $C(X)$. For any $\mu, \nu \in M(X)$, define 
	\begin{equation*}
		\rho (\mu, \nu) = \sum_{i=1}^{\infty} \frac {|\int f_i \mathrm{d}{\mu} - \int f_i \mathrm{d}{\nu} |}{2^i ( \|f_i \| + 1)},
	\end{equation*}
	where $\|f_i\| = \max \{|f_i(x)|: x \in X \}$ for all $i \in \mathbb{N^+}$.
	  Then $ \rho$ is a metric on $M(X)$ and the induced topology is also the weak*-topology on $M(X)$ (see for example \cite{PW},Theorem 6.4). Thus, $W(\cdot, \cdot)$ and $\rho(\cdot, \cdot)$ are compatible.

	\subsubsection{\bf Invariant measures}
	We say $\mu \in M(X)$ is $G$-invariant if $\mu (g^{-1} B) = \mu (B)$ holds for all $B \in \mathcal{B}(X)$ and $g \in G$. Denote by $M(X,G)$ the collection of all $G$-invariant Borel probability measures. In the weak$^*$-topology, $M(X,G)$ is a nonempty compact convex space.
	
	We say $B \in \mathcal{B}(X)$ is $G$-invariant if $g^{-1} B = B$ for any $g \in G$.
	And $\mu \in M(X,G)$ is ergodic if for any $G$-invariant Borel set $B \in \mathcal{B}(X)$, $\mu (B) = 0$ or $\mu(B) = 1$ holds. Denote by $M^e (X,G)$ the collection of all ergodic measures on $(X,G)$. It is well known that $M^e (X,G)$ is the collection of all extreme points of $M(X,G)$. Thus, $M^e (X,G)$ is nonempty. Using Choquet representation theorem, for each $\mu \in M(X,G)$ there is a unique measure $\tau$ on the Borel subsets of the  compact space $M(X,G)$ such that $\tau (M^e(X,G)) =1$ and $\mu = \int_{M^e (X,G)} m \mathrm{d} \tau (m)$, which is called the ergodic decomposition of $\mu$.

	We say $(X,G)$ is uniquely ergodic if $M^e (X,G)$ is singleton. Since $M^e (X,G)$ is the set of extreme points of $M(X,G)$, then $(X,G)$ is uniquely ergodic if and only if $M(X,G)$ is singleton.

Similar to Birkhoff pointwise ergodic, Lindenstrauss estabilished the pointwise ergodic theorem on $G$-systems \cite{EL}. 	A left F{\o}lner sequence $ \mathcal{F} = \{F_n \}_{n=1}^{\infty} $ of $G$ is \emph{tempered} if  there exists some $C > 0$ such that 
\begin{equation*}
| \bigcup_{k < n} F_k^{-1} F_n | \le C | F_n |\text{ for all } n \in \mathbb{N}^+.
\end{equation*}
And for every left F{\o}lner sequence, there is a subsequence which is tempered.

   \begin{lemma}[Pointwise Ergodic Theorem]\label{lem1}
   	Let $(X,G)$ be a $G$-system and $\mu$ be a $G$-invariant Borel probability measure. Suppose $ \mathcal{F} = \{F_n \}_{n=1}^{\infty} $ is a tempered F{\o}lner sequence of $G$. Then for any $f \in L^1(X,\mu)$, there is a G-invariant $f^* \in L^1(X,\mu)$ such that 
   	\begin{equation*}
   		\lim_{n \to +\infty} \frac 1{|F_n|} \sum_{g\in F_n} f (g x) = f^* (x) \ \ a.e.
   	\end{equation*} 
   	In particular, if $\mu$ is ergodic, one has
   	\begin{equation*}
   		\lim_{n \to +\infty} \frac 1{|F_n|} \sum_{g\in F_n} f (g x) = \int f(x) \mathrm{d} \mu (x) \ \ a.e.
   	\end{equation*} 
   \end{lemma}

    Given a left F{\o}lner sequence $\mathcal{F} = \{ F_n \}_{n=1}^{\infty}$ of $G$ and $x \in X$, we have $\{ \mu_{x, F_n} \}_{n=1}^{\infty} \subset M(X)$. Denote by $M_x^{\mathcal{F}}$ the collection of all limit points of $\{ \mu_{x, F_n} \}_{n=1}^{\infty}$. Since $M(X)$ is compact, we have $M_x^{\mathcal{F}} \not = \emptyset$. Moreover, $M_x^{\mathcal{F}} \subset M(X,G)$. If $M_x^{\mathcal{F}} = \{ \mu \}$, say $x$ is a generic point of $\mu$ with respect to 
    $\mathcal{F}$. And the pointwise ergodic theorem shows that if $\mu \in M^e (X,T)$ and $\mathcal{H}$ is a tempered F{\o}lner sequence of $G$, then the set of all generic points of $\mu$ with respect to   $\mathcal{H}$ has full measure.

    \section{Proof of Theorem \ref{thm1.2}}\label{4}

    In this section, we will prove that $(X,G)$ is weak mean equicontinuous with respect to a left F{\o}lner sequence $\mathcal{F}$ of $G$ if and only if $(X \times X, G)$ is mean equicontinuous with respect to $\mathcal{F}$.
    
    \begin{proof}[Proof of Theorem \ref{thm1.2}] 	Let $(X,G)$ be a $G$-system and $\mathcal{F}=\{F_n\}_{n\in\mathbb{N}}$ be a left F{\o}lner sequence of $G$. Firstly we assume that $(X,G)$ is mean  equicontinuous with respect to $\mathcal{F}$. Then for given $\varepsilon>0$, there exists $\delta>0$ such that 
    	\[
    	\limsup_{n \to +\infty} \frac 1{|F_n|} \sum_{g\in F_n} d(gx, gy) < \varepsilon,
    	\]
    whenever $x,y \in X$ with $d(x,y) < \delta$. 
    Hence for $(x_1,y_1),(x_2,y_2)\in X\times X$ satisfy $d(x_1,x_2)<\delta$ and $d(y_1,y_2)<\delta$, since $\mu_{((x_1,y_1,),(x_2,y_2)),F_n} \in \Pi (\mu_{(x_1, y_1), F_n}, \mu_{(x_2, y_2), F_n})$ for any $n \in \N^+$, 
     one has
    \[
    \begin{split}
    	W_\mathcal{F}((x_1,y_1),(x_2,y_2)) 
    	& = \limsup_{n \to \infty}  W (\mu_{(x_1,y_1),F_n} , \mu_{(x_1,y_1),F_n} ) \\
    	& \le \limsup_{n \to \infty}  \int \tilde d\left((x,y),(x',y')\right)d\mu_{((x_1,y_1,),(x_2,y_2)),F_n} \\
    	& = \limsup_{n \to \infty} \frac 1{|F_n| } \sum_{g \in F_n}  \tilde d \big( g (x_1,y_1), g (x_2,y_2)\big) \\
    	& = \limsup_{n \to +\infty} \frac 1{|F_n|} \sum_{g\in F_n} \left(d(gx_1, gx_2)+d(gy_1, gy_2)\right) \\
    	& < 2\varepsilon.
    \end{split}
    \]
    This implies that $(X\times X,G)$ is weak mean  equicontinuous with respect to $\mathcal{F}$.
    
    Now assume that  $(X\times X,G)$ is weak mean  equicontinuous with respect to $\mathcal{F}$. Then for given $\varepsilon>0$, there exists $\delta>0$ such that 
    if  $(x_1,y_1),(x_2,y_2)\in X\times X$ satisfy $d(x_1,x_2)<\delta$ and $d(y_1,y_2)<\delta$, then
    $W_\mathcal{F}((x_1,y_1), (x_2,y_2))< \varepsilon.$ 
   Hence for $x,y \in X$ with $d(x,y) < \delta$, one has 
\begin{align*}
\epsilon>W_\mathcal{F}((x,y),(y,y))&=\limsup_{n\to\infty}W(\mu_{(x,y),F_n}, \mu_{(y,y),F_n})\\
&\ge \limsup_{n\to\infty}\int_{X\times X} d(x',y')d\mu_{(x,y),F_n}(x',y')\\
&= \limsup_{n \to +\infty} \frac 1{|F_n|} \sum_{g\in F_n} d(gx, gy).
\end{align*}
Therefore, $(X,G)$ is  mean  equicontinuous with respect to $\mathcal{F}$. This ends the proof of Theorem \ref{thm1.2}.
    	\end{proof}

    \section{Proof of Theorem \ref{thm1.4} and Theorem \ref{thm1.5}}\label{5}
   In this section, we study uniquely ergodic points and uniform convergence. Firstly we show some properties of uniquely ergodic points, which are useful to prove Theorem \ref{thm1.4} and Theorem \ref{thm1.5}. 
    \begin{lemma}\label{lem6}
    	Let $(X,G)$ be a $G$-system and $x \in X$ be a uniquely ergodic point.  Then for any left F{\o}lner sequence $\mathcal{F}$ of $G$, we have $\mu_{x,\mathcal{F}}$ exists and $ \mu_{x,\mathcal{F}} = \mu_x^G$.
    \end{lemma}
    
    \begin{proof}
    	Since $x$ is a uniquely ergodic point, then $M (\ol{Orb(x)}, G) = \{ \mu_x^G\}$. Thus, $M_x^{\mathcal{F}} = \{ \mu_x^G\}$, which implies $\mu_{x,\mathcal{F}}$ exists and $ \mu_{x,\mathcal{F}} = \mu_x^G$.
    \end{proof}
    
     \begin{lemma}\label{12}Let $(X,G)$ be a $G$-system and $\mathcal{F}$ be a left F{\o}lner sequence of $G$. Then the following two statements are equivalent:
    	\begin{itemize}
    		\item[(\emph{1})]  $(X,G)$ is uniquely ergodic;
    		\item[(\emph{2})] $W_\mathcal{F}(x,y)=0$ for all $x, y\in X$.
    	\end{itemize}
    \end{lemma}
\begin{proof}(\emph{1}) $\Rightarrow$ (\emph{2}) 
	Let $\mu$ be a uniquely ergodic measure on $(X,G)$. Given $x,y \in X$, by Lemma \ref{lem6} we have $\mu_{x, \mathcal{F}} = \mu_{y, \mathcal{F}}  = \mu$. Thus,
	\[
	W_\mathcal{F}(x,y)= W (\mu_{x, \mathcal{F}}, \mu_{y, \mathcal{F}}) = W (\mu, \mu) = 0.
	\]
	
	(\emph{2}) $\Rightarrow$ (\emph{1})  
	Let $ \mathcal{H}$ be a tempered subsequence of $\mathcal{F}$. Assume that  $(X,G)$ is not uniquely ergodic, then there exist distinct ergodic measures $\mu$ and $\nu$ in $M(X,G)$. By pointwise ergodic theorem, there is $x\in X$  (resp. $y \in X$) which is generic point of $\mu$ (resp. $\nu$)
	with respect to $ \mathcal{H}$. Then 
	$$W_\mathcal{F}(x,y)\ge W_\mathcal{H}(x,y)=W(\mu,\nu)>0.$$
	This is impossible. Therefore, $(X,G)$ is uniquely ergodic. 
	\end{proof}

    \begin{lemma}\label{13}
    	Let $(X,G)$ be a  weak  mean  equicontinuous $G$-system with respect to a left F{\o}lner sequence  $\mathcal{F}$ of $G$. For $x\in X$, the following two statements are equivalent:
    	\begin{itemize}
    		\item[(\emph{1})] $x$ is a uniquely ergodic point;
    		\item[(\emph{2})] $W_\mathcal{F}(x,gx)=0$ for all $g\in G$.
    	\end{itemize}
    	\end{lemma}
    \begin{proof}
    	(\emph{1}) $\Rightarrow$ (\emph{2}) is immediately from Lemma \ref{12}.
    	
    	(\emph{2}) $\Rightarrow$ (\emph{1})
    	 Since $W_\mathcal{F}(x,gx)=0$ for all $g\in G$ and $(X,G)$ is weak  mean  equicontinuous with respect to $\mathcal{F}$, one has  $W_\mathcal{F}(x,y)=0$ for all $y\in \overline{Orb(x)}$. Then for all $y_1,y_2\in \overline{Orb(x)}$, one has 
    	$$W_\mathcal{F}(y_1,y_2)\le W_\mathcal{F}(x,y_1)+W_\mathcal{F}(x,y_2)=0.$$ Hence, $W_\mathcal{F}(y_1,y_2)=0$ for all $y_1,y_2\in \overline{Orb(x)}$.  By Lemma \ref{12}, $(\overline{Orb(x)},G)$ is  uniquely ergodic. Hence, $x$ is a uniquely ergodic point.
    	\end{proof}
    
        \begin{lemma}\label{132}
        	Let $(X,G)$ be a  weak  mean  equicontinuous $G$-system with respect to a  left F{\o}lner sequence  $\mathcal{F}$ of $G$. If $\mu\in M^e(X,G)$, then $(\emph{supp}\mu,G)$  is uniquely ergodic. Thus all the points in $\emph{supp}\mu$ are uniquely ergodic points. 
    \end{lemma}
\begin{proof}
	Assume that $(\text{supp}\mu,G)$ is not uniquely ergodic. Then there exists $\nu\in M^e(X,G)$ with $\mu \not= \nu$ such that 
	$$\text{supp}\nu\subset \text{supp}\mu.$$
	 Let $ \mathcal{H}$ be a tempered subsequence of $\mathcal{F}$. Then there exist $x_m,y$ such that
	 $$\mu_{x_m,\mathcal{H}}=\mu\text{ for any }m\in\mathbb{N}^+, \mu_{y,\mathcal{H}}=\nu\text{ and }\lim_{m\to\infty} x_m=y.$$
	 Thus,
	 $$W_\mathcal{F}(x_m,y)\ge W_\mathcal{H}(x_m,y)=W(\mu,\nu)>0\text{ for any }m\in\mathbb{N}^+.$$
	This contradicts with  the condition that $(X,G)$ is weak  mean  equicontinuous with respect to $\mathcal{F}$. Hence, $(\text{supp}\mu,G)$ is uniquely ergodic.
	\end{proof}

Now we prove Theorem \ref{thm1.4}.
\begin{proof}[Proof of Theorem \ref{thm1.4}]
	
	We will show that (\emph{1}) $\Rightarrow$ (\emph{2}) $\Rightarrow$ (\emph{3}) $\Rightarrow$ (\emph{4}) $\Rightarrow$ (\emph{1}).
	
	    (\emph{1}) $\Rightarrow$ (\emph{2}) is obvious.

		(\emph{2}) $\Rightarrow$ (\emph{3}) 
		 Assume that $(X,G)$ is weak  mean  equicontinuous with respect to a left F{\o}lner sequence $\mathcal{F}$. Lemma \ref{132} shows that all points in $\cup_{\mu\in M^e(X,G)}\text{supp}\mu$ are uniquely ergodic points. And by Lemma \ref{13}, one has $W_\mathcal{F}(x,gx)=0$ for every $g\in G$ and $x\in \cup_{\mu\in M^e (X,G)}\text{supp}\mu$. Then by the weak equicontinuity, one has $W_\mathcal{F}(x,gx)=0$ for all $g\in G$ and $x\in \overline{\cup_{\mu\in M^e(X,G)}\text{supp}\mu}$. Thus, all points in $\overline{\cup_{\mu\in M^e(X,G)}\text{supp}\mu}$ are uniquely ergodic points by  Lemma \ref{13}.
		 
		 For $x \in X$, one has $x \in \overline{\cup_{\mu\in M(X,G)}\text{supp}\mu} = \overline{\cup_{\mu\in M^e(X,G)}\text{supp}\mu}$, which shows that $x$ is a uniquely ergodic point. Thus by Lemma \ref{lem6}, one deduces that $\mu_{x,\mathcal{F}} $ exists and $\mu_{x,\mathcal{F}} = \mu_x^G$.
The  continuity of the map $x\to \mu_x^G$ is immediately from the assumption that $(X,G)$ is weak  mean  equicontinuous.

	(\emph{3}) $\Rightarrow$ (\emph{4}) 
	Given a left F{\o}lner sequence $\mathcal{F}=\{F_n\}_{n=1}^\infty$ of $G$, one has that
	$$\lim_{n\to\infty}\mu_{x,F_n}=\mu_x^G\text{ for any }x\in X.$$
		To our aim, it is enough to show that the convergence is uniformly. Assume that the convergence is not uniformly.  Then there exists $\varepsilon>0,x_i \in X$ and $n_i\to\infty$ such that  
		$$\rho(\mu_{x_i,F_{n_i}},\mu_{x_i}^G) \ge \varepsilon\text{ for all }i\in\mathbb{N}^+.$$
		Passing by a subsequence, we can assume that 
		$$\lim_{i \to \infty} x_i = x^* \ \text{and} \ \lim_{i\to\infty}\mu_{x_i,F_{n_i}} = \mu.$$
    	It is clear that $\mu\in M (X,G)$ and
    	\[
    	\rho (\mu, \mu_{x^*}^G)  =  \rho ( \lim_{i \to \infty} \mu_{x_i,F_{n_i}} , \lim_{i \to \infty}  \mu_{x_i}^G) = \lim_{i \to \infty} \rho(\mu_{x_i,F_{n_i}},\mu_{x_i}^G) \ge \varepsilon.
    	\]
    	By the ergodic decomposition theorem, there exists $\tilde \mu\in M^e(X,G)$ such that 
		$$\rho(\tilde\mu,\mu_{x^*}^G)\ge \varepsilon\text{ and }\text{supp}\tilde \mu\subset \text{supp} \mu.$$
		Fix $y\in \text{supp}\tilde \mu$, then $\tilde \mu = \mu_y^G$. And there exists $y_i\in\overline{Orb(x_i)},i\in\mathbb{N}^+$ such that 
	\begin{align}\label{1234}\lim_{i\to\infty}y_i=y.\end{align} 
		For $i \in \N^+$, since $\overline{Orb(x_i)}$ is uniquely ergodic, we have $\mu^G_{x_i} = \mu^G_{y_i} $. Thus,
			$$\limsup_{i\to\infty}\rho(\mu_{y_i}^G,\mu_y^G)=\limsup_{i\to\infty}\rho(\mu_{x_i}^G,\mu_y^G)=\rho(\mu_{x^*}^G,\tilde \mu)\ge \varepsilon.$$ 
			Combining this with \eqref{1234}, the map $z\to\mu_z^G$ is not continuous which is  contradictory to  	(\emph{3}). Thus, (\emph{3}) $\Rightarrow$ (\emph{4}) is valid. 
			
			(\emph{4}) $\Rightarrow$ (\emph{1}) 
			Given a left F{\o}lner sequence $\mathcal{F} = \{F_n\}_{n=1}^{\infty}$ of $G$, then by (4) for any $x \in X$, $\mu_{x,\mathcal{F}}$ exists. Now we assume that $(X,G)$ is not weak mean equicontinuous with respect to $\mathcal{F}$. Then there are $\varepsilon > 0$ and $y_m , y \in X$ with $y_m \to y$ such that
			$$W_{\mathcal{F}} (y_m, y) > \varepsilon.$$ This implies that $W (\mu_{y_m, \mathcal{F}}, \mu_{y, \mathcal{F}}) > \varepsilon$ for any $m \in \N^+$. Since $W(\cdot, \cdot)$ and $\rho (\cdot, \cdot)$ are compatible, there is $\delta > 0$ such that $\rho (\mu_{y_m, \mathcal{F}}, \mu_{y, \mathcal{F}}) > \delta$ for any $m \in \N^+$. Passing by a subsequence, there exist  $\delta'\in(0,\infty)$ and $f\in C(X)$ such that
			$$|\int_Xfd\mu_{y,\mathcal{F}}-\int_Xfd\mu_{y_m,\mathcal{F}}|>\delta'\text{ for all }m\in\mathbb{N}^+.$$			
			Then 
				$$|f^*(y)-f^*(y_m)|>\delta'\text{ for all }m\in\mathbb{N}^+.$$
			Since $\lim_{m\to\infty}y_m=y$, one has $f^*$ is not continuous. This contradicts with the assumption.  Hence, $(X,G)$ is weak mean equicontinuous.
	\end{proof} 

    To prove Theorem \ref{thm1.5}, we need only to show (\emph{2}) $\Rightarrow$ (\emph{3}), for the proof of  (\emph{3}) $\Rightarrow$ (\emph{4}) $\Rightarrow$ (\emph{1}) $\Rightarrow$ (\emph{2}) is as same as the proof of Theorem \ref{thm1.4}.

    \begin{proof}[Proof of Theorem \ref{thm1.5}]
    	
    	(\emph{2}) $\Rightarrow$ (\emph{3}) 
   
    	Fix $x\in X$ and assume that $\mathcal{L}=\{L_n\}_{n=1}^\infty$. Then for $g \in G$, by Theorem \ref{AN} one has 
    	$$W_\mathcal{L}(x,gx)=\limsup_{n\to\infty}W(\mu_{x,L_n},\mu_{gx,L_n})=\limsup_{n\to\infty} \inf_{ h \in Aut(L_n)} \frac 1{|L_n|} \sum_{l \in L_n} d(l x, l^h g x) = 0,$$
    	where $\text{Aut}(F)$ is the  permutation group of $F$, and the last equality comes from that $\mathcal{L}$ is right F{\o}lner sequence,
    	Then by  Lemma \ref{13}, one has that $x$ is a uniquely ergodic point. With the same reason in Theorem \ref{thm1.4}, one has that map $x \to \mu_x^G$ is continuous.
    	\end{proof}	
    
     \appendix
    \renewcommand{\appendixname}{Appendix~\Alph{section}}
    \section{}\label{A}
    
    In this section, we will prove our definition of weak mean  equicontinuity is the same as in \cite{ZZ}. To our aim, we need some definitions and lemmas which are stated as follow.
    
    Let $(X,G)$ be a $G$-system and $\mathcal{F}=\{F_n\}_{n=1}^\infty$ be a left F{\o}lner sequence  of $G$. For any $x,y\in X$, define
    $$W_\mathcal{F}^{\text{Aut}}(x,y)=\limsup_{n\to\infty}\inf_{h\in \text{Aut}(F_n)} \frac 1{|F_n|}\sum_{g\in F_n} d(gx,g^hy),$$
    where $\text{Aut}(F)$ is the  permutation group of $F$ and $g^h=h(g)$.
    
    Our main result in this section is the following.
    \begin{thm}\label{AN}Let $(X,G)$ be a $G$-system and $\mathcal{F}=\{F_n\}_{n=1}^\infty$ be a left F{\o}lner sequence  of $G$. Then $W_\mathcal{F}^{\text{Aut}}(x,y)=W_\mathcal{F}(x,y)$ for every $x,y\in X$.
    \end{thm} 
    
   \begin{defn}
   	An $n \times n$ matrix $A = (a_{ij}) : i,j = 1, 2, \cdots, n $ is called doubly stochastic provided it is non-negative and the sum of entries in every row and every column is 1. And denote by $DM(n)$ the set of all $n \times n$ doubly stochastic matrices.
   \end{defn}
   
   Given $A \in DM(|F_n|)$. For $z_1, z_2 \in X$, put
   \[
   \pi_A (z_1, z_2) = \frac 1{|F_n|} \sum_{g_1 x = z_1, g_2 y = z_2 } A(g_1,g_2).
   \]
   Then $\pi_A \in \Pi (\mu_{x,F_n},\mu_{y,F_n})$. 
   Define $\Psi : DM (| F_n |) \la \Pi (\mu_{x,F_n},\mu_{y,F_n})$ such that $\Psi (A) = \pi_A$. It is easy to check that $\Psi$ is surjective.

   \begin{defn}
   	Let $\sigma \in S_n$, where $S_n$ is the $n$-order permutation group. The permutation matrix $A^{\sigma}$ is the $n \times n$ matrix $A^{\sigma} = (a^{\sigma}_{ij}) : i,j =1,2, \cdots, n$, defined as follows:
   	\begin{equation*}
   		a^{\sigma}_{ij} =
   		\begin{cases}
   			1 \ \ \text{if} \ \ \sigma(j) = i, \\
   			0 \ \ \text{otherwise}.
   		\end{cases}
   	\end{equation*}
   	And denote by $PM(n)$ the set of all $n \times n$ permutation matrices.
   \end{defn}
   
   For $A \in PM(|F_n|)$, there is $h \in Aut(F_n)$ such that 
   $
   \Psi (A) = \frac 1{|F_n|} \sum_{g \in F_n} \delta_{(gx, g^h y)}.
   $
   On the other hand, for $h \in Aut (F_n)$, let $\pi_h = \frac 1{|F_n|} \sum_{g \in F_n} \delta_{(gx, g^h y)}$ and
   \[
   A_h (g_1, g_2) =
   \begin{cases}
   	1 \ \text{ if } g_2 = g_1^h \\
   	0 \ \text{ otherwise}.
   \end{cases}
   \]
  Then $A_h \in PM (|F_n|)$ and $\Psi (A_h) = \pi_h$. Thus,
   \begin{equation*}
   	\Psi \big( PM (|F_n|) \big) = \{\frac 1{|F_n|} \sum_{g \in F_n} \delta_{(gx, g^h y)} : h \in Aut (F_n) \}.
   \end{equation*}
   
   The relation between $DM(n)$ and $PM(n)$ was shown by Birkhoff and Von Neumann (see for example \cite{AB}, Chapter 2, Theorem 5.2).
   \begin{lemma}[Birkhoff-Von Neumann Theorem]
   	The extreme points of $DM(n)$ are exactly the $n \times n$ permutation matrices $PM(n)$.
   \end{lemma}

Now we prove Theorem \ref{AN}.
   \begin{proof}[\bf Proof of Theorem \ref{AN}]
   	Since $\Psi : DM (| F_n |) \la \Pi (\mu_{x,F_n},\mu_{y,F_n})$ is surjective, one has
   \begin{equation*}
   	W(\mu_{x,F_n},\mu_{y,F_n}) = \inf_{A \in DM(|F_n|)} \sum_{g_1, g_2 \in F_n} \frac {A(g_1,g_2) d (g_1 x, g_2 y)}{|F_n|}.
   \end{equation*}
  	Then by Birkhoff-Von Neumann Theorem, we deduce that
  	\begin{equation*}
  		\begin{split}
  			W(\mu_{x,F_n},\mu_{y,F_n}) 
  			&= \inf_{A \in PM(|F_n|)} \sum_{g_1, g_2 \in F_n} \frac {A(g_1,g_2) d (g_1 x, g_2 y)}{|F_n|}  \\
  			&= \inf_{ h \in Aut(F_n)} \frac 1{|F_n|} \sum_{g \in F_n} d(g x, g^h y),
  			\end{split}
  	\end{equation*}
  	the last equality comes from $\Psi \big( PM (|F_n|) \big) = \{\frac 1{|F_n|} \sum_{g \in F_n} \delta_{(gx, g^h y)} : h \in Aut (F_n) \}$. 
  	This ends the proof of Theorem \ref{AN}.
  \end{proof}


\begin{thebibliography}{99}
    	\bibitem{EJK}
    	\newblock E. Akin, J. Auslander and K. Berg,
    	\newblock When is a transitive map chaotic? 
    	\newblock \emph{Convergence in Ergodic Theory and Probility, de Gruyter, Berlin}, Ohio State Univ. Math. Res. Inst. Publ., de Gruyter, Berlin, \textbf{5} (1996), 25--40.
    	
       	\bibitem{AB} 
       \newblock A. Barvinok,
       \newblock \emph{A Course in Convexity},
       \newblock  Graduate Studies in Mathematics, \textbf{54}, American Mathematical Society, Providence, RI, 2002. x+366 pp.
    	
    	\bibitem{EF}
    \newblock E. F{\o}lner,  
    \newblock On groups with full Banach mean value,
    \newblock \emph{Math. Scand.}, \textbf{3} (1955), 243--254.	
    
        \bibitem{FB}
        \newblock N. Frantzikinakis and B. Host,  
        \newblock The logarithmic Sarnak conjecture for ergodic weights,
        \newblock \emph{Ann. of Math.} (2) \textbf{187} (2018), no. 3,  869--931.	
    
    	\bibitem{FG}
    	\newblock F. Garcia-Ramos,  
    	\newblock Weak forms of topological and measure-theoretical equicontinuty: relationships with discrete specturm and sequence entropy,
    	\newblock \emph{Ergodic Theory Dynam. Systems}, \textbf{39} (2019), no. 2, 729--746.
    	
    	\bibitem{GM}
    	\newblock F. Garcia-Ramos and B. Marcus,  
    	\newblock Mean sensitive, mean equicontinuous and almost periodic functions for dynamical systems,
    	\newblock \emph{Discrete Contin. Dyn. Syst.}, \textbf{37} (2017), no. 4, 1211--1237.
    	
    	
    	
    		
    	\bibitem{PJ}
    	\newblock P. Halmos and J. Von Neumann,
    	\newblock Operator methods in classical mechanics, \uppercase\expandafter{\romannumeral2},
    	\newblock \emph{Ann. of Math.} (2) \textbf{43} (1942), 332--350.
    	
    	
    	
    	\bibitem{WJLX}
    	\newblock W. Huang, J. Li, J. Thouvenot, L. Xu and X. Ye,  
    	\newblock Bounded complexity, mean equicontinuity and discrete spectrum, 
    	\newblock \emph{Ergodic Theory Dynam. Systems}, To Appear.
    	
    	    	\bibitem{JM}
    	\newblock R. Johnson and J. Moser,  
    	\newblock The rotation number for almost periodic potentials,
    	\newblock \emph{Comm. Math. Phys.}, \textbf{84} (1982), no. 3, 403--438.
    	
    	\bibitem{JL}
    	\newblock J. Li, 
    	\newblock How chaotic is an almost mean equicontinuous system? 
    	\newblock \emph{Discrete Contin. Dyn. Syst.}, \textbf{38} (2018), no. 9, 4727--2744.
    	
    	  	\bibitem{EL}
    	\newblock E. Lindenstrauss,  
    	\newblock Pointwise theorems for amenable groups,
    	\newblock \emph{Invent. math.}, \textbf{146} (2001), 259--295.
    	
    	\bibitem{JSX}
    	\newblock J. Li, S. Tu and X. Ye,  
    	\newblock Mean equicontinuity and mean sensitivity,
    	\newblock \emph{Ergodic Theory Dynam. Systems}, \textbf{35} (2015), no. 9, 2587--2612.
    	
    	\bibitem{MA}
    	\newblock B. Michael and F. Alexander,  
    	\newblock Ergodic theorems for coset spaces,
    	\newblock \emph{J. Anal. Math.}, \textbf{35} (2018), no. 1, 85--122.
    	
    	\bibitem{JCO}
    	\newblock J. C. Oxtoby,  
    	\newblock Ergodic sets,
    	\newblock \emph{Bull. Amer. Math. Soc.}, \textbf{58} (1952), 116--136.
    	
    	\bibitem{CV}
    	\newblock C. Villani,  
    	\newblock \emph{Topics in Optimal Transportation},
    	\newblock AMS, Providence, RI, 2003.
	
    	
    	
    \bibitem{PW} 
    \newblock P. Walters,
    \newblock \emph{An Introduction to Ergodic Theory},
    \newblock  Graduate Texts in Mathematics, \textbf{79}, Springer-Verlag, New York-Berlin, 1982.
    	
    
    
  
    	
    	\bibitem{ZZH}
    	\newblock M. Zhang and Z. Zhou, 
    	\newblock Uniform ergodic theorems for discontinuous skew-product flows and applications to {S}chr\"{o}dinger equations,
    	\newblock \emph{Nonlinearity}, \textbf{24} (2011), no. 5, 1539--1564.
    
    
    	\bibitem{ZZ}
    	\newblock Z. Zheng and L. Zheng,  
    	\newblock A new metric for statistical properties of long time behaviors,
    	\newblock \emph{J. Differential Equations},
    	\textbf{269} (2020), no. 4, 2741--2773.
    \end{thebibliography}
\end{document}